\documentclass[11 pt]{amsart}
\usepackage{amsmath}
\usepackage{graphicx}
\usepackage[colorlinks= true, urlcolor=black, citecolor=blue]{hyperref}

\newtheorem{theorem}{Theorem}[section]
\newtheorem{lemma}[theorem]{Lemma}
\newtheorem{corollary}{Corollary}[theorem]
\newtheorem{proposition}{Proposition}[section]

\theoremstyle{definition}
\newtheorem{definition}[theorem]{Definition}
\newtheorem{example}[theorem]{Example}

\newtheorem{Problem}{Problem}
\theoremstyle{remark}

\numberwithin{equation}{section}

\title{Acylindrical hyperbolicity of Subgroups}

\author{ Abhijit Pal}
\address{Indian Institute of technology, Kanpur}
\email{\href{mailto:abhipal$@$iitk.ac.in}{abhipal$@$iitk.ac.in}}
\author{Rahul Pandey}
\address{Indian Institute of technology, Kanpur}
\email{\href{mailto:prahul$@$iitk.ac.in}{prahul$@$iitk.ac.in}}



\begin{document}
\begin{abstract}
      Suppose  $G$ is a finitely generated group and $H$ is a
subgroup of $G$. Let $\partial_{c}^{\mathcal{F}\mathcal{Q}}G$ denote the contracting boundary of $G$ with
the topology of fellow travelling quasi-geodesics defined by Cashen-Mackay \cite{cashen2017}.
   In this article, we show that if the limit set $\Lambda(H)$ of $H$ in
$\partial_{c}^{\mathcal{F}\mathcal{Q}}G$ is compact and contains at least three points then
the action of the subgroup $H$ on the
space of distinct triples $\Theta_{3}(\Lambda(H))$ is properly discontinuous. By applying a result of B. Sun \cite{BinSun},
if  the limit set $\Lambda(H)$ is compact and the action of $H$ on $\partial_{c}^{\mathcal{F}\mathcal{Q}}G$
is non-elementary  then $H$ becomes an acylindrically hyperbolic group
\end{abstract}
\maketitle
\tableofcontents


\section{Introduction}
The Gromov boundary of a word hyperbolic group is now a well studied object and it plays a crucial role in bordification of finitely generated groups.  Recent efforts have been made
to define a boundary of a geodesic space which is a generalization of the Gromov boundary. A geodeisc ray
$\gamma$ in a metric space is said to be \textit{Morse} if for any $K\geq 1,\epsilon\geq 0$
there exists a constant $N=N(K,\epsilon)$ such that any $(K,\epsilon)$-quasi-geodesic
with end points on $\gamma$ lie in the $N$-neighborhood of $\gamma$.
Cordes in \cite{Morse2015} defined a boundary of a proper geodesic space by taking
all asymptotic Morse geodesic rays starting from a fixed point and it was called  Morse
boundary.
Cordes in \cite{Morse2015} equipped the boundary with direct limit topology
motivated by the contracting boundary of CAT$(0)$ spaces defined
by Charney and Sultan in \cite{Charney15}. Direct limit topology on the Morse
boundary has several drawbacks, in general it is not even first countable
and hence not metrizable. To rectify this situation Cashen-Mackay in
\cite{cashen2017} introduced a new topology on the Morse boundary which was called
topology of fellow travelling quasi-geodesics.  They
worked with contracting quasi-geodesic rays from a fixed base point to define a boundary at infinity and it was called to be
contracting boundary. As a set both Morse boundary and contracting boundary are same
but with two topologies, the topology  of fellow travelling quasi-geodesics is coarser than the direct limit topology.
Cashen-Mackay, in \cite{cashen2017}, showed that  the Morse boundary of a
finitely generated group  with the topology of fellow travelling quasi-geodesics is metrizable.

\par

Osin in \cite{DOsin} introduced the notion of acylindrically hyperbolic groups.
An action of a group $G$ on a metric space $(X,d)$ is said to be
\textit{acylindrical} if for every $\epsilon>0$ there exists $R,N>0$
such that if $d(x,y)>R$ then the set $\{g\in G:d(x,gx)<\epsilon\mbox{ and } d(y,gy)<\epsilon\}$
contains at most $N$  elements.
A group $G$ is called \textit{acylindrically hyperbolic}
if it admits an acylindrical action on a hyperbolic metric space $X$ such that the limit set of $G$ on the Gromov boundary $\partial X$  contains at least three points.
Recently,  Sun in \cite{BinSun} gave a dynamical characterization of
acylindrically hyperbolic groups motivated by the works of Bowditch \cite{Bowditch}, Freden\cite{freden}, Tukia\cite{Tukia}) and Yaman \cite{Yaman}.
 An action of a group $G$ on a compact metrizable space $M$ by
homeomorphism is called a
convergence group action if the induced diagonal action on the space of distinct triples
$$ \Theta_{3}(M)=\{(x_{1},x_{2},x_{3}) \in M^{3}\ |\ x_{1}\neq x_{2},x_{2}\neq
x_{3},x_{1}\neq x_{3}\}$$ is properly discontinuous.
The action of a group $G$ on a metric space $M$ is said to be $elementary$ if it fixes a
subset of $M$ with at most two elements; otherwise the action is called $non$-$elementary$. Sun, in \cite{BinSun}, proved that a group $G$ having non-elementary convergence group action on some compact metrizable space $M$ is
acylindrically hyperbolic (See Corollary 1.3 of \cite{BinSun}).\par

 We denote the contracting boundary of a finitely generated group $G$ with fellow
travelling quasi-geodesics topology by $\partial_{c}^{\mathcal{F}\mathcal{Q}}G$. $\partial_{c}^{\mathcal{F}\mathcal{Q}}G$ is metrizable (Corollary\ 8.6 of
\cite{cashen2017}).
 Let $H$ be a subgroup of $G$ and take the limit set
$\Lambda(H)$ of $H$ in
  $\partial_{c}^{\mathcal{F}\mathcal{Q}}G$.
 Each element of group $G$ acts on  $\partial_{c}^{\mathcal{F}\mathcal{Q}}G$ by
homeomorphism and each element of $H$ leaves $\Lambda(H)$ invariant.
 Thus $H$ acts homeomorphically on $\Lambda(H)$. In this article, we will analyze the  action of the subgroup $H$  on $\Lambda(H)$ when $\Lambda(H)$ is
a compact subset of
 $\partial_{c}^{\mathcal{F}\mathcal{Q}}G$.
 \begin{theorem}\label{main}
Suppose $G$ is a finitely generated group. Let $H$ be a subgroup of $G$
such that the limit set  $\Lambda(H)$ in $\partial_{c}^{\mathcal{F}\mathcal{Q}}G$ contains at
least three elements. If $\Lambda(H)$ is compact subset then the action of the subgroup $H$ on the
space of distinct triples $\Theta_{3}(\Lambda(H))$ is properly discontinuous.
\end{theorem}
If the the action of the subgroup $H$ on $\partial_{c}^{\mathcal{F}\mathcal{Q}}G$  is non-elementary then the limit set $\Lambda(H)$ of $H$ in $\partial_{c}^{\mathcal{F}\mathcal{Q}}G$ contains at least three points.
 One immediate corollary we can deduce from Theorem \ref{main} by applying Sun's result (Corollary 1.3 of \cite{BinSun})
is the following:
 \begin{corollary}\label{main cor}
Let $H$ be a subgroup  of a finitely generated group $G$ such that
the limit set $\Lambda(H)$  is compact in contracting boundary $\partial_{c}^{\mathcal{F}\mathcal{Q}}G$ of $G$. If the action of the subgroup $H$ on $\partial_{c}^{\mathcal{F}\mathcal{Q}}G$  is non-elementary  then $H$ is an acylindrically hyperbolic group.
 \end{corollary}

 The action of a word hyperbolic group $G$ on its Gromov boundary $\partial G$ is convergence and hence any subgroup of a word hyperbolic group has convergence action on $\partial G$. A quasiconvex subgroup $H$ of a word hyperbolic group $G$ is itself word hyperbolic and the limit set $\Lambda(H)$, being homeomorphic to $\partial H$, is  compact. An example of a non-hyperbolic subgroup of a word hyperbolic group with limit set compact can be obtained from Rips construction in \cite{rips}.
 \begin{example} \label{trivial 1}
 Rips in \cite{rips} showed that given any infinite finitely presented group $Q$, there exists a short exact sequence
       $$1 \rightarrow N\rightarrow G \rightarrow Q \rightarrow 1,$$
       such that $G$ is word hyperbolic group and $N$ is finitely generated non hyperbolic group.
       The limit set $\Lambda(N)$, being closed, is compact. In fact, as $N$ is a normal subgroup of $G$, $\Lambda(N)= \partial G$.  \\
\end{example}
  \begin{example}\label{trivial 2}
  In the Example \ref{trivial 1}, the subgroup $N$ of the word hyperbolic group $G$ is finitely generated.
       In \cite{DasMj}, Das and Mj. constructed an infinitely generated malnormal subgroup $K$ of a free group  $F_{2}$ of rank $2$,
       whose limit set $\Lambda(K)$ is whole $\partial F_{2}$. See section 3 of \cite{DasMj} for a detailed discussion about this example. The limit set  $\Lambda(K)$ is compact.
  \end{example}
  By the work of Olshanski, Osin \& Sapir in \cite{lac}, there exists
     a   finitely generated group $G$ (Tarski Monster Group) which is  not virtually cyclic, every proper subgroup of $G$ is cyclic, Morse
     and hence $\partial_{c}^{\mathcal{F}\mathcal{Q}}G\neq\phi$ (Theorem 1.12 of  \cite{lac}). The group $G$ is not acylindrcally hyperbolic but  all infinite subgroups of $G$, being cyclic, are elementary.
 \begin{example}\label{trivial 3} Let $A$ be a Tarski Monster Group in which every proper subgroup is infinite cyclic  and $\partial_{c}^{\mathcal{F}\mathcal{Q}}A\neq\phi$ (See Theorem 1.12 of  \cite{lac}). Let $M$ be a proper infinite subgroup of $A$. Then $M$ is cyclic and $\Lambda(M)$ consists of two points.
   By Rips construction in \cite{rips} given any infinite finitely presented group $Q$, there exists a short exact sequence
       $1 \rightarrow N\rightarrow B \rightarrow Q \rightarrow 1$
       such that $B$ is a word hyperbolic group and $N$ is a finitely generated non-hyperbolic group.
       The limit set $\Lambda(N)$ is compact.  Consider the group $G=A*B$ and the subgroup $H=M*N$ of $G$.
       The subgroup $H$ is contained in the hyperbolic group $M*B$ and the subgroup $M*B$ is Morse in $G$.
       Hence, the limit set $\Lambda(H)$ is compact in the contracting boundary $\partial_{c}^{\mathcal{F}\mathcal{Q}}G$ of $G$.
       The group $G$ is acylindrically hyperbolic but not a hyperbolic group.
        \end{example}
  In the Examples \ref{trivial 1}, \ref{trivial 2} \& \ref{trivial 3}, the subgroups are non-elementary and are contained in a hyperbolic group. The limit sets  of these subgroups, being contained in the Gromov boundary of a hyperbolic group, are compact. It is a trivial fact that non-elementary subgroups of a hyperbolic group are acylindrically hyperbolic and one does not need Corollary \ref{main cor} to prove it.
  Thus, the above examples are not new and not much interesting. In \cite{cashen2017}, Cashen-Mackay proved that a finitely generated group $G$ has compact contracting boundary $\partial_{c}^{\mathcal{F}\mathcal{Q}}G$ if and only if the group $G$ is hyperbolic. By using
  the ideas of proof in Theorem \ref{main},   we will give another proof of the fact that a group $G$ is hyperbolic  if
  its contracting boundary  $\partial_{c}^{\mathcal{F}\mathcal{Q}}G$  is compact (See Theorem \ref{hyp theorem}).
    If $H$ is a Morse subgroup
  of $G$ then $\Lambda(H)$ is homeomorphic to $\partial_{c}^{\mathcal{F}\mathcal{Q}}H$. Thus, for a Morse subgroup $H$ of a finitely generated group, the limit set $\Lambda(H)$ is compact if and only if $H$ is hyperbolic.
  To find examples which would make Theorem \ref{main} interesting, one should look for non-Morse subgroups in
  non-hyperbolic groups or non-acylindrically hyperbolic groups with non-trivial Morse boundary.
  \begin{Problem} \label{prob 1}Find a finitely generated group $G$ and a subgroup $H$ of $G$ such that the followings hold: \\
  (i) $G$ is not acylindrically hyperbolic and $\partial_{c}^{\mathcal{F}\mathcal{Q}}G\neq \phi$,\\
  (ii) the limit set $\Lambda(H)$ is compact in $\partial_{c}^{\mathcal{F}\mathcal{Q}}G$ and contains at least three points,\\
  (iii) $H$ is not contained in any hyperbolic subgroup of $G$.
  \end{Problem}

  By the work of Dahmani-Guiradel-Osin \cite{dah-gui-osin} and Sisto \cite{sisto}, it is known that an acylindrically hyperbolic
  group always contain non-elementary hyperbolic Morse subgroups and hence contracting boundary of an acylindrically hyperbolic
  group is always non-empty. Let $G$ be a finitely generated group and $H$ be a subgroup of $G$. If the limit set $\Lambda(H)$
  of $H$ in $\partial_{c}^{\mathcal{F}\mathcal{Q}}G$ is a compact set and the action of the subgroup $H$ on $\partial_{c}^{\mathcal{F}\mathcal{Q}}G$  is non-elementary then, by Corollary \ref{main cor}, the subgroup $H$
  is acylindrically hyperbolic and hence $\partial_{c}^{\mathcal{F}\mathcal{Q}}H\neq \phi$.
  \begin{Problem}\label {prob 2} Does there exists a finitely generated group $G$ and a subgroup $H$ of $G$ such that $\partial_{c}^{\mathcal{F}\mathcal{Q}}H=\phi$ but the limit set $\Lambda(H)$ is a non-empty subset of $\partial_{c}^{\mathcal{F}\mathcal{Q}}G$ ?
  \end{Problem}
  The groups $\mathbb Z\oplus\mathbb Z$ and Baumslag-Solitar group $BS(1,2)$ have empty contracting boundaries. In \cite{drutu},
  Drutu and Sapir proved that if no asymptotic cone of a group contains a cut point then that group has an empty contracting boundary.
  Consider the standard embedding $\mathbb Z\oplus\mathbb Z\hookrightarrow PSL(2,\mathbb Z[i])$ defined on generators as
  $$(1,0)\mapsto \begin{pmatrix}1 & 1\\ 0&1\end{pmatrix}, (0,1)\mapsto \begin{pmatrix}1 & i\\ 0&1\end{pmatrix}.
  $$
The group $ PSL(2,\mathbb Z[i])$ is a discrete subgroup of $PSL(2,\mathbb C)$ which is orientation preserving isometry group
of hyperbolic 3-space $\mathbb H^3$. The limit set of $\mathbb Z\oplus\mathbb Z$ in the boundary at infinity $\partial \mathbb H^3$
of $\mathbb H^3$ is the singleton set $\{\infty\}$. The contracting boundary $\partial_{c}^{\mathcal{F}\mathcal{Q}}(PSL(2,\mathbb Z[i]))$
 is non-empty but the limit set $\Lambda(\mathbb Z\oplus\mathbb Z)$ in $\partial_{c}^{\mathcal{F}\mathcal{Q}}(PSL(2,\mathbb Z[i]))$
is empty. \par
Let $G$ be a group hyperbolic relative to $H$ (in the sense of Bowditch). The group $G$ acts properly discontinuously on a proper hyperbolic metric space $X_G$ and $G$ has convergence \& geometrically finite action on the Gromov boundary $\partial X_G$. The boundary $\partial X_G$ of $X_G$ is called the Bowditch boundary or relatively hyperbolic boundary of $G$ and is denoted by $\partial_{rel}G$. Consider the group $G=\mathbb Z*(\mathbb Z\oplus\mathbb Z)$ and the subgroup $H=\mathbb Z\oplus\mathbb Z$ of $G$.
The group $G$ is hyperbolic relative to $H$. Here, the limit set $\Lambda_{rel}(H)$ of $H$ in $\partial_{rel}(G)$ is a singleton set
whereas $\partial_{c}^{\mathcal{F}\mathcal{Q}}H=\phi$ and the limit set $\Lambda(H)=\phi$ in $\partial_{c}^{\mathcal{F}\mathcal{Q}}G$.
For a group $G$ hyperbolic relative to $H$ with $\partial_{c}^{\mathcal{F}\mathcal{Q}}H=\phi$,
Cashen-Mackay \cite{cashen2017} proved that $\partial_{c}^{\mathcal{F}\mathcal{Q}}G$ is embedded in the set of all conical limit points for the action of $G$ on $\partial_{rel}G$.

Let $G$ be a finitely generated group and $H$ be a subgroup of it. To understand the relationship between $\partial_{c}^{\mathcal{F}\mathcal{Q}}H$ and the limit set $\Lambda(H)$, it is worth to study whether there exists a continuous extension
$\partial_c i:\partial_{c}^{\mathcal{F}\mathcal{Q}}H\to \partial_{c}^{\mathcal{F}\mathcal{Q}}G$ of the inclusion $i:H\hookrightarrow G$. Such continuous
extensions are called Cannon-Thurston maps, the name  derived from the work of J.Cannon and W.Thurston in \cite{Can-Thu}. The Cannon-Thurston maps were extensively studied by Mahan Mj., we refer the reader to the introduction section of the article \cite{Mj Ann}  by Mj. for a detailed discussion of results and problems oriented with Cannon-Thurston maps.

\begin{example} (i) Let $\Sigma$  be a closed orientable surface of genus $2$ and $f:\Sigma\to\Sigma$ be a pseudo-Anosov homeomorphism.
Consider the mapping torus $M_f$$=$$\frac{\Sigma\times [0,1]}{\{(x,0)\sim(f(x),1): ~x\in\Sigma\}}$. The fundamental groups $\pi_1(\Sigma)$
and $\pi_1(M_f)$ are hyperbolic groups. The group $\pi_1(M_f)$ is the HNN extension $<\pi_1(\Sigma),t~| ~tat^{-1}=f_*(a)~\mbox{for all }a\in\pi_1(\Sigma)>$, where $f_{*}$ is the isomorphism on $\pi_1(\Sigma)$ induced from $f$.  By the work of Cannon and Thurston, in \cite{Can-Thu}, it follows that
there exists a Cannon-Thurston map for the embedding $\pi_1(\Sigma)\hookrightarrow \pi_1(M_f)$ (See Section 4 of \cite{Can-Thu}, see also Theorem 4.3 of \cite{Mitra} for a more general statement).
Now consider the group
Let $H=<\pi_1(\Sigma),x,y~ | ~xy=yx>$ and $G=<\pi_1(M_f),x,y~|~xy=yx>$. Both the groups $H$ and $G$ are hyperbolic relative to $\mathbb Z\oplus\mathbb Z$. By Theorem 2.10 of \cite{abhijit}, there exists a continuous extension $\partial_{rel}i:\partial_{rel}H\to\partial_{rel}G$
of the inclusion $i:H\hookrightarrow G$. The contracting boundary $\partial_{c}^{\mathcal{F}\mathcal{Q}}H$ of $H$ is embedded in $\partial_{rel}H$
and  the contracting boundary $\partial_{c}^{\mathcal{F}\mathcal{Q}}G$ of $G$ is embedded in $\partial_{rel}G$ (See Theorem 7.6 of \cite{cashen2017}). In this example, $\partial_{rel} i(\partial_{c}^{\mathcal{F}\mathcal{Q}}H)$ is contained in $\partial_{c}^{\mathcal{F}\mathcal{Q}}G$.
To see this fact, let us take
a contracting geodesic ray $\alpha:[0,\infty)\to H$ in $H$ with $\alpha(0)=e$, where $e$ is the identity element of $H$. The ray $\alpha$ being contracting implies that the intersections
of $\alpha$ with the left cosets of $\mathbb Z\oplus\mathbb Z$ are uniformly bounded. Thus there exists a number $D_{\alpha}\geq 0$ and
disjoint subintervals $[s_n,t_n]$ with $t_n<s_{n+1}$ such that $\alpha([s_n,t_n])$ lie in a left coset of $\mathbb Z\oplus\mathbb Z$, $\alpha([t_n,s_{n+1}]))$ lie in a left coset of $\pi_1(\Sigma)$ in $H$ and length of every $\alpha|_{[s_n,t_n]}$ is at most $D_{\alpha}$.
Let $u_n\in (t_n,s_{n+1})$, the sequence $\{\alpha(u_n)\}$ converges to $\alpha(\infty)$ in $H\cup \partial_{c}^{\mathcal{F}\mathcal{Q}}H$. As Cannon-Thurston map exists,
the sequence $\{i(\alpha(u_n))\}$ converges to $\partial_c i(\alpha(\infty))$ in $G\cup\partial_{rel}G$. Let $\beta_n$ be a geodesic in $G$
joining $e$ and $i(\alpha(u_n))$. Then $\beta_n$ is of the form
$[e,\alpha(s_1)]\cup\alpha([s_1,t_1])\cup [\alpha(t_1),\alpha(s_2)]\cup...\cup[\alpha(t_n),\alpha(u_n)]$, where $[\alpha(t_{j}),\alpha(s_{j+1})]$
and $[\alpha(t_n),\alpha(u_n)]$ are geodesics in left cosets of $\pi_1(M_f)$. The geodesics in left cosets of $\pi_1(M_f)$ are uniformly contracting as
$\pi_1(M_f)$ is a hyperbolic group. Note that any quasi-geodesic in $G$ joining $e$ and $i(\alpha(u_n))$ passes through points $\alpha(s_j),\alpha(t_j)$. As length of every $\alpha|_{[s_n,t_n]}$ is at most $D_{\alpha}$ and each geodesics in left cosets of $\pi_1(M_f)$ are uniformly contracting the geodesic $\beta_n$ is uniformly contracting for every $n$. The sequence $\{\beta_n\}$ of geodesics converge to the
uniformly contracting geodesic $\beta:[0,\infty)\to G$ where $\beta([0,\infty))=[e,\alpha(s_1)]\cup\alpha([s_1,t_1])\cup [\alpha(t_1),\alpha(s_2)]\cup...$. Thus, $\partial_{rel} i(\alpha(\infty))=\beta(\infty)\in \partial_{c}^{\mathcal{F}\mathcal{Q}}G$. For this example,
there exists a continuous extension  $\partial_c i:\partial_{c}^{\mathcal{F}\mathcal{Q}}H\to \partial_{c}^{\mathcal{F}\mathcal{Q}}G$ of  $i:H\hookrightarrow G$ and  the restriction of the map $\partial_{rel} i$ on $\partial_{c}^{\mathcal{F}\mathcal{Q}}H$ is same as  $\partial_c i$.
Note that for the inclusion
$j:\pi_1(\Sigma)\hookrightarrow G$, due to the existence of its Cannon-Thurston map,
$\partial_c j(\pi_1(\Sigma))=\Lambda(\pi_1(\Sigma))$. Hence, the limit set $\Lambda(\pi_1(\Sigma)) $
is compact in $ \partial_{c}^{\mathcal{F}\mathcal{Q}}G$. \\ \\
(ii) Consider the group
$G=<\pi_1(\Sigma),x,y,t~ | ~xy=yx,tx=xt,ty=yt,tat^{-1}=f_*(a)~\forall ~a\in\pi_1(\Sigma)>.$
The group $H$ is hyperbolic relative to $\mathbb Z\oplus\mathbb Z$ and the group $G$ is hyperbolic relative to $\mathbb Z\oplus\mathbb Z\oplus\mathbb Z$ (See Figure 1 of \cite{beeker}). By Theorem 2.10 of \cite{abhijit}, there exists a continuous extension $\partial_{rel}i:\partial_{rel}H\to\partial_{rel}G$
of the inclusion $i:H\hookrightarrow G$. The inclusion $i$ preserves parabolic limit points. Hence, $\partial_{rel} i(\partial_{c}^{\mathcal{F}\mathcal{Q}}H)$ is contained in the set of conical limit points for the action of $G$
on $\partial_{rel}G$. But here $\partial_{rel} i(\partial_{c}^{\mathcal{F}\mathcal{Q}}H)$ may not be contained in $ \partial_{c}^{\mathcal{F}\mathcal{Q}}G$, as for a geodesic $\gamma$ in $G$ connecting two points of a contracting geodesic ray in $H$, the intersection of $\gamma$ with left cosets of  $\mathbb Z\oplus\mathbb Z\oplus\mathbb Z$ may not be uniformly bounded. The groups $H$ and $G$ here are also $CAT(0)$ groups with isolated flats. By changing the relatively hyperbolic boundary with  visual boundary of $CAT(0)$ spaces on which the groups act, Beeker, Cordes, Gardam, Gupta, Stark in \cite{beeker} showed that a Cannon-Thurston map does not exist for the pair $(G,H)$ (See Theorem 5.6 of \cite{beeker} for a more general statement).
\end{example}

\begin{Problem}
Find an example of a finitely generated group $G$ and a non-Morse subgroup $H$ of $G$ such that $G$ is neither hyperbolic nor relatively hyperbolic and  a Cannon-Thurston map  $\partial_c i:\partial_{c}^{\mathcal{F}\mathcal{Q}}H\to \partial_{c}^{\mathcal{F}\mathcal{Q}}G$ exists for the inclusion $i:H\hookrightarrow G$.
\end{Problem}
\noindent\textit{Acknowledgements:} We thank the anonymous referee for his/her valuable comments and suggestions which has helped in improving the exposition of this article from an earlier draft. Research of the first author was supported by DST-INSPIRE Grant IFA12-MA-19.

\section{Morse \& Contracting Quasi-geodesics}
\begin{definition}
\begin{enumerate}
\item(Quasi-isometry): Let $K\geq 1,\epsilon\geq 0$. Let $(X,d_{X})$ and $(Y,d_{Y})$ be two metric spaces.
 A map $f: X \rightarrow Y$ is said to be a $(K,\epsilon)$-$quasi$-$isometric$ embedding if
$$\frac{1}{K}d_{X}(a,b)-\epsilon\leq d_{Y}(f(a),f(b))\leq Kd_{X}(a,b)+\epsilon$$ for all $a,b\in X.$
In addition, if for each $y\in Y$ there exists $x\in X$ such that $d_Y(y,f(x))\leq K$ then $f$
is said to be a $(K,\epsilon)$-$quasi$-$isometry$ between $X$ and $Y$.
\item (Quasi-geodesic): Let $X$ be metric space. A map $c: I \rightarrow X$,
where $I$ is any interval in $\mathbb{R}$ with usual metric, is
said to be $(K,\epsilon)$-$quasi$-$geodesic$ if c is a $(K,\epsilon)$-quasi-isometric embedding.
\end{enumerate}
\end{definition}

\begin{definition}
(Morse quasi-geodesic): A quasi-geodesic $\gamma$ in a geodesic metric space X is
called $N$-$Morse$ if
there exists a function $N:\mathbb{R}_{\geq 1} \times \mathbb{R}_{\geq 0} \rightarrow
\mathbb{R}_{\geq 0}$ such that if $q$ is
any $(K,\epsilon)$-quasi-geodesic with endpoints on $\gamma$ then
$q$ lies in the closed $N(K,\epsilon)$-neighborhood of $\gamma$.
We call $N$ to be the $Morse\ Gauge$ of $\gamma$.
\end{definition}

\noindent {A function $\rho$ is called \textit{sublinear} if it is non-decreasing, eventually
non-negative, and $lim_{r \rightarrow \infty} \rho(r)/r = 0$.}
\begin{definition}
(Contracting quasi-geodesic):
Let $\gamma:I\rightarrow X$ be a quasi-geodesic in a geodesic metric space $X$.
Let $\pi_{\gamma} : X \rightarrow \mathbb P({\gamma(I)})$ be defined as
$\pi_{\gamma}(x)=\{z \in \gamma \ |d(x, z) = d(x, \gamma(I))\}$, where $\mathbb P({\gamma(I)})$ denotes the power set of $\gamma(I)$.
The map $\pi_{\gamma}$ is called to be the closest point
projection to $\gamma(I)$.  For a sublinear function $\rho$, we say that $\gamma$
is $\rho$-$contracting$ if for all $x$ and $y$ in X:
 $$ d(x,y)\leq d(x,\gamma(I))\implies
diam(\pi_{\gamma(I)}(x)\cup\pi_{\gamma(I)}(y))\leq \rho(d(x,\gamma(I))) .$$
 We say that a quasi-geodesic $\gamma$ is contracting if it is $\rho$-$contracting$ for
some sublinear function $\rho$.
\end{definition}
\textbf{Note:} In the above two definitions one can take any subset $Z$ of $X$
instead of quasi-geodesics and then we have the definitions of Morse and contracting subsets.
\begin{theorem}  (Theorem 1.4 of \cite{arzhan}) \label{Cont} Let $Z$ be a subset of a geodesic
metric space $X$. The followings are equivalent:
\begin{enumerate}
\item  $Z$ is Morse.
\item $Z$ is contracting.
\end{enumerate}
\end{theorem}

Examples of contracting (or Morse) quasi-geodesics include
quasi-geodesics in hyperbolic spaces, axis of pseudo-Anosov mapping classes in
Teichm{\"u}ller space \cite{Minsky} etc.

\section{Contracting Boundary and Topology on it}
\begin{definition}\label{Morse Constant}
Given a sublinear function $\rho$ and constants $L \geq 1$ and $A\geq 0$, define:
       $$ k(\rho,L,A):=max\{3A,3L^{2},1+\inf\{R>0|\mbox{ for all } r\geq
R,3L^{2}\rho(r)\leq r\} \}$$
       Define:
       $$ k'(\rho,L,A):= (L^{2}+2)(2k(\rho,L,A)+A).$$
\end{definition}

\textit{Notation:} If $f$ and $g$ are functions then we say $f \preceq g$ if there
exists a constant $C > 0$ such
that $f(x) \leq Cg(Cx + c) + C$ for all $x$ . If $f \preceq g$ and $f \succeq g$
then we write $f\asymp g$.

\begin{lemma}(Lemma 6.3 of \cite{arzhan})\label{lemma1}
      Given a sublinear function $\rho$ and a constant $C \geq 0$ there exists a
sublinear function
      $\rho' \asymp \rho$ such that if $Z \subseteq X$ and $Z' \subseteq X$ have
Hausdorff distance at most $C$ and $Z$ is $\rho$–$contracting$ then $Z'$ is
$\rho'$–$contracting$.
\end{lemma}
\begin{lemma}(Lemma 3.6 of \cite{cashen2017})\label{Clemma}
Given a sublinear function $\rho$ there is a sublinear function $\rho \asymp \rho'$
such that every subsegment of a $\rho$-$contracting$ geodesic is $\rho'$-$contracting$.
\end{lemma}

\begin{lemma}(Lemma 4.4 of \cite{cashen2017})\label{Keyprop1}  Suppose
$\alpha:[0,\infty)\to X$ is a continuous, $\rho$-contracting, $(L, A)$-quasi-geodesic ray and
$\beta:[0,\infty)\to X$ is a continuous $(L, A)$-quasi-geodesic ray in $X$ such that $d(\alpha(0),
\beta(0)) \leq k(\rho, L, A)$. If
there are $r, s \in [0, \infty)$ such that $d(\alpha({r}), \beta({s})) \leq k(\rho, L,
A)$ then the Hausdorff distance $d_{Haus}(\alpha[0,r],\beta[0,s]) \leq k'(\rho, L, A)$.
If $\alpha[0,\infty)$ and $\beta[0,\infty)$ are asymptotic then their Hausdorff
distance is at most $k'(\rho, L, A)$.
\end{lemma}

\begin{lemma}\label{CM 4.6} (Lemma 4.6 of \cite{cashen2017}) Let $\alpha$ be a $\rho$ contracting geodesic ray, and let $\beta$ be a continuous $(L,A)$ quasi-geodesic with $\alpha_{0}=\beta_{0}=o$. Given some numbers $r$ and $J$, suppose there exists a point $x\in \alpha$ with $d(x, o)\geq R$ and $d(x, \beta)\leq J$. Let $y$ be the last point on the subsegment of $\alpha$ between $o$ and $x$ such that $d(y,\beta)=\kappa(\rho,L,A)$. Then there exists a function $\lambda(\phi,p,q)$ defined for sublinear function $\phi$, $p\geq1$ and $q\geq 0$ such that $\lambda$ is monotonically increasing in $p,q$ and:
$$d(x,y)\leq 2J+\lambda(\phi,L,A). $$ Thus:
                                         $$d(o,y)\geq R-2J-\lambda(\rho,L,A)$$
\end{lemma}

\noindent{In the statement of Lemma 4.6 of \cite{cashen2017}, $\beta$ was taken to be a quasi-geodesic ray while the same proof works for any finite quasi-geodesic segment. }
\begin{definition}
(Contracting Boundary, $\partial_{c}X$): Let $X$ be a proper geodesic metric space
with basepoint $o$.
Define $\partial_{c}X$ to be the set of contracting quasi-geodesic rays based at $o$
modulo Hausdorff equivalence. \end{definition}

\begin{proposition}(Lemma 5.2 of \cite{cashen2017}) For each $\zeta \in \partial_{c}X$:
     \begin{enumerate}
     \item The set of contracting geodesic rays in $\zeta$ is non-empty.
     \item There is a sublinear function:
     $$ \rho_{\zeta}(r):=\underset{\alpha,x,y}{sup}\ diam\big( \pi_{\alpha}(x)\cup
\pi_{\alpha}(y) \big)$$
     Here the supremum is taken over geodesics $\alpha \in \zeta$ and points x and y
such
     that $d(x, y) \leq d(x, \alpha) \leq r$
     \item Every geodesic in $\zeta$ is $\rho_{\zeta}$-$contracting$.
     \end{enumerate}
\end{proposition}
$Notation:$ Let $N^{c}_{r}o=\{x\in X|\ d(o,x)\geq r\}$ for a metric space $(X,d)$.
 \begin{definition}
 \textit{ (Topology of fellow travelling quasi-geodesics, Cashen-Mackay, Definitions 5.3, 5.4 of \cite{cashen2017}):}
     Let $X$ be a proper geodesic metric space. Take  $\zeta \in \partial_{c}X$. Fix
a geodesic ray $\alpha^{\zeta} \in \zeta$.
    For each $r \geq 1$, define $U(\zeta, r)$ to be the set of points  $\eta \in
\partial_{c}X$ such that for all
    $L \geq 1$ and $A \geq 0$ and every continuous $(L,A)$-quasi-geodesic ray
$\beta \in \eta$ we have
    \begin{equation}
    \tag{*}
         d(\beta, \alpha^{\zeta} \cap N^{c}_{r}o) \leq k(\rho_{\zeta},L,A).
    \label{eqn:key1}
         \end{equation}
     Define a topology on $\partial_{c}X$ by
 \begin{equation}
\tag{**}
\mathcal{F}\mathcal{Q}:=\{U \subset \partial_{c}X\ |\ \text{for all}\ \zeta \in U,
\text{there exists}\ r \geq 1, U(\zeta, r)\subset U \}
\label{eqn:key2}
\end{equation}
     The contracting boundary equipped with this topology is called
\textit{topology of fellow travelling quasi-geodesics} by Cashen-Mackay
\cite{cashen2017} and is
     denoted by $\partial^{\mathcal{F}\mathcal{Q}}_{c}X$. \end{definition}\par

   A \textit{bordification} of a Hausdorff topological space $X$ is a a Hausdorff space $\overline{X}$
   containing $X$ as an open, dense subset.
     The contracting boundary of a proper geodesic metric space provides a
bordification of $X$ by $\overline{X}:= X \cup \partial_{c}X$  as follows (Definition 5.14 of \cite{cashen2017}):\\
     For $x \in X$, we take a neighborhood basis for $x$ to be metric balls about $x$.

     For $\zeta \in \partial_{c}X$, we take a neighborhood basis for $\zeta$ to be the sets
$\widehat{U}(\zeta, r)$ consisting of $U(\zeta, r)$ and
     points $x \in X$ such that $d(\gamma,N^{c}_{r}o \cap \alpha^{\zeta} )
\leq k(\rho_{\zeta} ,L,A)$ for every $L\geq 1, A\geq 0$, and continuous $(L, A)$-quasi-geodesic segment $\gamma$ with endpoints $o$ and $x$.

     Consider
$\mathcal{T} := \{ U \subseteq X\cup \partial_{c}X\ |\ \forall x \in X\cap U, \exists~ r\geq 0, B(x;r)\subseteq U\mbox{ and~}
\forall~ \zeta \in U\cap  \partial_{c}X,  \exists~ r\geq 1, \widehat{U}(\zeta,r) \subseteq U \}$.
$\mathcal{T}$ defines a topology on $\overline{X}$ and $\overline{X} $ with respect to this topology $\mathcal{T}$,
 is first countable, Hausdorff and $X$ is dense open set in $\overline{X}$. (See Proposition 5.15 of \cite{cashen2017})

     \begin{proposition} (Proposition 5.15 of \cite{cashen2017})
          $\overline{X}:= X \cup \partial_{c}X$ topologized as above defines a first
countable bordification of $X$ such that the
          induced topology on $\partial_{c}X$ is the topology of fellow-travelling
quasi-geodesics.
     \end{proposition}

\begin{definition}
(Limit set, $\Lambda(G)$): \label{def}  If $G$ is a finitely generated group acting
properly discontinuously
on a proper geodesic metric space $X$ with basepoint $o$ we define the $limit\ set\
\Lambda(G) :=\overline{Go} \backslash Go$,
the topological frontier of orbit of $o$ under the $G$-action in
$\bar{X}$.\end{definition}

\begin{definition}(Definition 3.11 of \cite{MCordes2017},
Asymptotic and Bi-asymptotic geodesics)  Let
$\gamma: (-\infty,\infty)\rightarrow X$ be a bi-infinite geodesic in $X$ with $\gamma(0)$
a closest point to basepoint $o$ along $\gamma$. Let $\zeta$ $\in \partial_{c}X$. We say $\gamma$ is \textit{forward asymptotic} to $\zeta$ if for any
contracting geodesic ray $\gamma^{\zeta}:[0,\infty) \rightarrow X$ representing $\zeta$ with $\gamma^{\zeta}(0)=o$, there exists $K>0$ such that
$$d_{Haus}(\gamma([0,\infty)),\gamma^{\zeta}([0,\infty))) < K$$
We define \textit{backwards asymptotic} similarly. If $\gamma$ is both forward and backward asymptotic to $\zeta$
and $\eta$ respectively, then we say $\gamma$ is bi-asymptotic to the ordered pair $(\zeta,\eta)$. If $\gamma$ is bi-asymptotic to $(\zeta,\eta)$ then we will call $\gamma$ to be bi-infinite geodesic joining $\zeta,\eta$
and we denote it by $[\zeta,\eta]$.

\end{definition}

 The following lemma says that any two points of contracting boundary can be joined by a contracting bi-infinite geodesic.
 See sections 3.3 and 3.4 of \cite{MCordes2017} for a detailed discussion about limiting geodesics.
 \begin{lemma}(Lemma 3.4 of \cite{MCordes2017})
  Let $X$ be proper geodesic metric space and $\zeta$, $\eta$ be two distinct points in $\partial_{c}X$. Then there exists a contracting
  bi-infinite geodesic joining $\zeta$ and $\eta$.
\end{lemma}

 \begin{theorem}\label{Mtheorem}
    Let $X$ be a proper geodesic metric space with $\partial_{c}X\neq \emptyset$.
Consider a sequence of  geodesics $\{\gamma_{n}\}$  with end points $\zeta_{n},\eta_{n}\in X\cup \partial_{c}X$.
     Suppose  $\zeta_{n}\rightarrow \zeta,\eta_{n}\rightarrow \eta$
in the topology of fellow travelling quasi-geodesics and $\zeta \neq \eta$. Then $\gamma_{n}$ passes through a bounded set.
    \end{theorem}
    \begin{proof}     Assume     $\zeta,\eta\in \partial_{c}X$.
    Let $k_{1}=k(\rho_{\zeta},3,0)$, $k_{2}=k(\rho_{\eta},3,0)$ and $k=max\{k_{1},k_{2}\}$.
     Fix a base point $o$ in $X$.
          For points $x,y\in X\cup \partial_{c}X$, we denote $[x,y]$ by a geodesic segment, geodesic ray or bi-infinite geodesic ray joining $x$ and $y$,
depending on whether $x$ and $y$, none, exactly one or both lies in $\partial_{c}X$.
If $\gamma$ is any parametrized path and $x,y \in \gamma$ then $[x,y]_{\gamma}$ stands for the segment of $\gamma$ between $x$ and $y$. \\
    Let $\gamma$ be a bi-infinite geodesic joining $\eta$ and $\zeta$ with contracting function $\rho$.
    Let $p_{n}$ and $p$ be  nearest point projections from $o$ to $\gamma_{n}$ and $\gamma$ respectively. We will prove that the
    sequence $\{d(o,p_{n})\}$ is bounded. \\
    Let $\alpha_{n}:=[o,p_{n}]\cup [p_{n},\zeta_{n}]_{\gamma_{n}}$, $\alpha:=[o,p]\cup [p,\zeta]_{\gamma}$, $\alpha'_{n}:=[o,p_{n}]\cup [p_{n},\eta_{n}]_{\gamma_{n}}$,
    and $\alpha':=[o,p]\cup
[p,\eta]_{\gamma}$. Paths $\alpha_{n}$, $\alpha'_{n}$, $\alpha$ and $\alpha'$ are continuous $(3,0)$ quasi-geodesics. We parametrize them by it's arc length.   \\
Since $\zeta_{n}\rightarrow \zeta,\eta_{n}\rightarrow \eta$ in
\textit{the topology of fellow travelling quasi-geodesics} \eqref{eqn:key1},\eqref{eqn:key2}, for any $r\geq1$ there exists $N(=N(r))$ such that
for all $n\geq N$
\begin{equation}\label{CM eqn}
    d(\alpha_{n},\alpha^{\zeta}([r,\infty)))\leq k \ \ ,\
d(\alpha'_{n},\alpha^{\eta}([r,\infty)))\leq k ,
   \end{equation}
   where $\alpha^{\zeta},\alpha^{\eta}$ are geodesic rays from $o$ representing $\zeta,\eta$ respectively.
   As  $\alpha,\alpha'$ are $(3,0)$ quasi-geodesics, by Proposition \ref{Keyprop1},
   $\alpha^{\zeta}$ and $\alpha$, $\alpha^{\eta}$ and $\alpha'$ lie in a bounded Hausdorff distance, say $M$, of each other.
   Then, as $\alpha, \alpha'$ are $(3,0)$-quasi-geodesics
   $$d(\alpha_n,\alpha([\frac{1}{3}(r-M),\infty)))\leq k+M,d(\alpha'_n,\alpha'([\frac{1}{3}(r-M),\infty)))\leq k+M.$$
   Let $k_1=k+M$ and $r_1=\frac{1}{3}(r-M)$.
   By Lemma \ref{Clemma} and Theorem \ref{Cont}, depending on the contracting function $\rho$ of $\gamma$, there exists a
   Morse gauge say $N_{\gamma}$ such that every subsegment of $\gamma$ is $N_{\gamma}$-Morse.
   We choose $r$ large such that $r_1>2k_1+3d(o,p)+N_{\gamma}(3,2k_{1})+1$.
  Thus, there exists a number $N=N(r)>0$ such that for any $n\geq N$, we have the following:
   \begin{equation}\label{closeness}
    d(\alpha_{n},\alpha ([r_1,\infty)))\leq k_1\ \ ,\
d(\alpha'_{n},\alpha'([r_1,\infty)))\leq k_1
   \end{equation}
   Thus, for $n\geq N$, there exist $x_n\in\alpha_n,y_n\in\alpha ([r_1,\infty)),x'_n\in\alpha'_n,y'_n\in\alpha'([r_1,\infty))$
   such that $d(x_n,y_n)\leq k_1$ and $d(x_n',y_n')\leq k_1$.\\ \\
   \textbf{Case 1:} Suppose $x_n\in[p_{n},\zeta_{n}]_{\gamma_{n}}$ and $x_n'\in [p_{n},\eta_{n}]_{\gamma_{n}}$.

   Path $[y_{n}',x_{n}']\cup [x'_{n},x_{n}]_{\gamma_{n}}\cup[x_{n},y_{n}]$ is $(1,2k_1)$-quasi-geodesic joining $y_{n}'$ and $y_{n}$.
   By Lemma \ref{Clemma}, $[y_{n}',y_{n}]_{\gamma}$ is $\rho'$- contracting, where $\rho' \asymp \rho$.
   By Theorem \ref{Cont}, the Hausdorff distance between $ [x'_{n},x_{n}]_{\gamma_{n}}$
   and $[y_{n}',y_{n}]_{\gamma}$ is  bounded, say it is $M_1$  and this
   bound $M_1$ is independent of $n$ and depends only on the contracting function $\rho$.  Now $p\in [y_{n}',y_{n}]_{\gamma} $,
   then there exists    a point say $z_{n}\in [x_{n}',x_{n}]_{\gamma_n} $  such that $d(p,z_{n})\leq M_1$. As $p_{n}$ is a nearest point projection of $o$ on
   $\gamma_{n}$, $d(o,p_{n})\leq M_1+d(o,p)$. \\ \\
   \textbf{Case 2:} Suppose  $x_{n}\in [o,p_{n}]$ and $x_{n}'\in [p_{n},\eta_{n}]_{\alpha_{n}'}$ for some number $n$.

    Consider the path $[x'_{n},p_{n}]_{\gamma_{n}}\cup [p_{n},x_{n}]_{\alpha'_{n}}$, it is a $(3,0)$-quasi-geodesic.
    As $d(x_{n},y_{n})\leq k_1$ and $d(x_{n}',y_{n}')\leq k_1$, the path
    $[y'_{n},x'_{n}]\cup [x'_{n},p_{n}]_{\gamma_{n}}\cup [p_{n},x_{n}]_{\alpha'_{n}}\cup [x_{n},y_{n}]$ is a $(3,2k_1)$-quasi-geodesic.
     Now $d(p,y_{n})> r_1 - d(o,p)$ and hence $d(y_{n},o)>r_1-2d(o,p)$. As $d(x_{n},y_{n})\leq k_1$, therefore $d(o,p_{n})\geq d(o,x_{n})> r_1-k_1-2d(o,p)$
     and hence $$d(p,p_{n})> r_1-k_1-3d(o,p).$$
     Thus, the $(3,0)$-quasi-geodesic path  $[x'_{n},p_{n}]_{\gamma_{n}}\cup [p_{n},x_{n}]_{\alpha'_{n}}$ lie outside the ball $B(p; r_1-k_1-3d(o,p))$.
     Again, as $d(x_{n},y_{n})\leq k_1$ and $d(x_{n}',y_{n}')\leq k_1$, the path
    $[y'_{n},x'_{n}]\cup [x'_{n},p_{n}]_{\gamma_{n}}\cup [p_{n},x_{n}]_{\alpha'_{n}}\cup [x_{n},y_{n}]$ lie outside the closed ball
    $\bar B(p; r_1-2k_1-3d(o,p))$. This says that the subsegment $\bar{B}(p;r_1-2k_1-3d(o,p))\cap \gamma$ is not $N_{\gamma}$ Morse, a contradiction.
    Thus, Case 2 is not possible. \\  \\
    \textbf{Case 3:} Now assume that $x_{n}\in [o,p_{n}]$ and $x_{n}'\in [o,p_{n}]$ for some number $n$. Then,
    following same argument as in first part of case 2  we get  $$d(p,x_{n})> r_1-k_1-d(o,p)~ \& ~d(p,x_{n}')> r_1-k_1-d(o,p)$$
    The path $[y'_{n},x'_{n}]\cup [x'_{n},x_{n}]_{\alpha_{n}'}\cup [x_{n},y_{n}]$ is a $(1,2k_{1})$ quasi-geodesic and it
    lie outside the closed ball $\bar B(p; r_1-2k_1-3d(o,p))$ which says that the segment $\bar{B}(p;r_1-2k_1-d(o,p))\cap \gamma$ is not $N_{\gamma}$-Morse, a contradiction.\\
    Thus, for all large $n $, $x_n\in[p_{n},\zeta_{n}]_{\gamma_{n}}$ and $x_n'\in [p_{n},\eta_{n}]_{\gamma_{n}}$ and we return to Case 1.\\
    If one of $\zeta$ or $\eta$ is in $X$ then correspondingly one of the end points of $\gamma_n$ lie in a bounded set.
    Hence, for all $n$, $\gamma_n$ passes through a bounded set. Also, note that
       the sequence $\{d(o,p_n)\}$  is bounded  if either $\zeta$ or $\eta$ or both  lie in $X$.

     \end{proof}

    \begin{corollary}\label{MainLemma}
       Let $\{\gamma_{n}\}$ be a sequence of  bi-infinite contracting geodesics with end points $\zeta_{n},\eta_{n}\in \partial_{c}X$.
      Suppose  $\zeta_{n}\rightarrow \zeta,\eta_{n}\rightarrow \eta$, $\zeta \neq \eta$
in the topology of fellow travelling quasi-geodesics. Let $\gamma$ be a bi-infinite contracting geodesic joining $\zeta$ and $\eta$. Then given any $x\in \gamma$, there exists a number $N$ such that for all $n\geq N$, $d(x, \gamma_{n})$ is less than some constant $K$,
where $K$ depends only on contracting function of $\gamma$.
    \end{corollary}
\begin{proof}
     Let $p_{n}$ and $p$ be a nearest point projections from $o$ to $\gamma_{n}$ and $\gamma$ respectively.
     We use the notations of Theorem \ref{Mtheorem}. Take $r_1> \max\{d(o,p)+d(p,x), 2k_1+3d(o,p)+N_{\gamma}(3,2k_{1})+1 \}$.
     There exists a number $N_1>0$ depending on $r$ such that for all $n\geq N_1$
    \begin{equation*}
    d(\alpha_{n},\alpha ([r_1,\infty)))\leq k_1 \ \ ,\
d(\alpha'_{n},\alpha'([r_1,\infty)))\leq k_1
   \end{equation*}
   Thus, for $n\geq N_1$, there exist $x_n\in\alpha_n,y_n\in\alpha ((r_1,\infty)),x'_n\in\alpha'_n,y'_n\in\alpha'((r_1,\infty))$
   such that $d(x_n,y_n)\leq k_1$ and $d(x_n',y_n')\leq k_1$.  We have proved in Theorem \ref{Mtheorem}
    that the sequence $\{d(o,p_{n})\}$ is bounded. Thus, there exists a number $N_2>0$ such that
    for all $n\geq N_2$, $x_n,x'_n$ does not lie in $[0,p_n]$. Let $N=\max\{N_1,N_2\}$. For all $n\geq N$, the geodesic $[x_n',x_n]_{\gamma_n}$
    lie in a bounded $K$-Hausdorff distance from $[y'_n,y_n]_{\gamma}$, where $K$ eventually depends only on contracting function of $\gamma$.
    Note that by the choice of $r_1$,
    $x\in [y'_n,y_n]_{\gamma}$. Thus, $d(x,\gamma_n)\leq d(x,[x_n',x_n]_{\gamma_n})\leq K$ for all $n\geq N$.
\end{proof}

\begin{lemma}\label{Convergence lemma} Let $X$ be proper geodesic metric space with non empty
contracting boundary and $o$ be a fixed base point in $X$.
Suppose $\{\zeta_{n}\}$ is a sequence in $\partial_{c}X$ converging to $\zeta$ in
the topology of fellow travelling quasi-geodesics. Suppose for each $n$,
 there exists a continuous
$(3,0)$-quasi-geodesic $\beta_n$ starting from $o$ and
representing  $\zeta_{n}$ such that $\beta_n$ is the concatenation of a geodesic $[o,p_n]$ and a geodesic ray $\nu_n$. Assume also that the sequence
$\{d(o,p_n)\}$ is bounded. Let $q_n$ be a point on
$\beta_{n}$ such that $d(o,q_n)\rightarrow \infty$. Then the following holds:\\
(i) Given any $r\geq 1$ there exists $N=N(r)$ such that for all $n\geq N$,
$$d(\gamma_n, \alpha^{\zeta}([r,\infty)))\leq k'(\rho_{\zeta},2L+1,A)$$
where $\gamma_n$ is any  continuous $(L,A)$-quasi-geodesic with end points $o$ and $q_{n}$ and  $\alpha^{\zeta}$ is a contracting geodesic representing $\zeta$ starting from the base point $o$.\\
(ii) The sequence $\{q_{n}\}$ converges to $\zeta$.
\end{lemma}
\begin{proof}(i) We first prove the following claim:\\
\textit{Claim}: Let  $\eta \in\partial_{c}X$ and $\beta$ be a continuous $(3,0)$-quasi-geodesic ray representing $\eta$
 such that $\beta$ is concatenation of a geodesic $[o,p]$ and a geodesic ray $\nu $ starting from $p\in X$. Then $\beta$ is contracting as it represents an element of the contracting boundary $\partial_{c}X$. We parametrize $\beta:[0,\infty)\to X$ by arc length
 where $\beta(0)=o$.  Consider a point $q\in\beta[0,\infty)$ which lies in $\nu$. Let $\gamma$ be any continuous $(L,A)$-quasi-geodesic joining $\beta(0)=o$ and $q$. Then there exists a $(2L+1,A)$-quasi-geodesic ray of the form $\gamma|_{[0,u]}*\mu$, where $\gamma(u)\in\gamma$ and $\mu$
 is a geodesic ray starting from $\gamma(u)$, such that $\gamma|_{[0,u]}*\mu$ is asymptotic to $\beta$ and $d(o,q)\leq 2d(o,p)+ d(o,\gamma(u))$. \\ \\
\textit{Proof of claim: }
 For all  $n$, there exist points $\gamma(t_n)\in\gamma$
 and $\beta(s_n)\in\beta([n,\infty))$ such that $d(\gamma,\beta|_{[n,\infty})=d(\gamma(t_n),\beta(s_n))$.
 If $\mu_n$ is a geodesic between $\gamma(t_n)$ and $\beta(t_n)$
 then $\gamma|_{[0,t_n]}*\mu_n$ is a $(2L+1,A)$-quasi-geodesic. An application of Arzela-Ascoli theorem tells that
 there exists a quasi-geodesic ray of the form $\gamma|_{[0,u]}*\mu$ where $\gamma(u)\in\gamma$ and $\mu$
 is a geodesic ray starting from $\gamma(u)$. Each $\gamma|_{[0,t_n]}*\mu_n$ lie $R$-neighborhood of $\beta$, where $R$
 depends on $L,A$ and the contraction function of $\beta$. Thus, $\gamma|_{[0,u]}*\mu$ also lie in the $R$-neighborhood
 of $\beta$ and hence it is asymptotic to $\beta$. Note that
 $d(p,q)=d(p,\beta(s_n))-d(q,\beta(s_n))\leq d(p,\beta(s_n))-d(\gamma(t_n),\beta(s_n))\leq d(p,\gamma(t_n))$.
   Thus, for all $n$, we have  $$d(o,q)\leq d(o,p)+d(p,q) \leq d(o,p)+d(p,\gamma(t_n))\leq2d(o,p)+ d(o,\gamma(t_n)).$$
Taking $n\to\infty$, we get $d(o,q)\leq 2d(o,p)+ d(o,\gamma(u))$. Hence, we have proved the claim.\\ \\
Let $\gamma_n:[0,s_n]\to X$ represents  the arc length re-parametrization  of the quasi-geodesic $\gamma_n$.
For each $n$, from the Claim it follows that there exists a $(2L+1,A)$-quasi-geodesic ray of the form ${\gamma_n}|_{[0,u_n]}*\mu_n$ asymptotic to $\beta_n$, where $u_n\in [0,s_n]$.
Thus, ${\gamma_n}|_{[0,u_n]}*\mu_n$ also represents $\zeta_n$. $\zeta_n\to\zeta$ implies that
$d({\gamma_n}|_{[0,u_n]}*\mu_n, \alpha^{\zeta}([r,\infty))\leq k(\rho_{\zeta},2L+1,A)$ for all large $n$.
There exists $x_n\in{\gamma_n}|_{[0,u_n]}*\mu_n$  such that $d(x_n,\alpha^{\zeta}([r,\infty))\leq k$.
If $x_n\in\mu_n$ then by Lemma \ref{Keyprop1}, $d(\gamma(u_n),\alpha^{\zeta})\leq k'$, where $k'=k'(\rho_{\zeta},2L+1,A)$ (as in Definition \ref{Morse Constant}).
Now, $d(o,q_n)\leq 2d(o,p_n)+ d(o,\gamma(u_n))$ and $\{d(o,p_n)\}$  bounded implies that
$d(o,\gamma_n(u_n))\to\infty$ as $d(o,q_n)\to\infty$. Thus, for all large $n$, $d(\gamma_n(u_n),\alpha^{\zeta}([r,\infty))\leq k'$.
If $x_n\in\gamma_n([0,u_n))$ then from definition $d(\gamma_n,\alpha^{\zeta}([r,\infty)))\leq k\leq k'$.\\ \\
(ii) Given $r\geq1$, our aim is to find large number $N$ such that for all $n\geq N$, the distance $d(\gamma_n, N^{c}_{r}o\cap \alpha^{\zeta})\leq \kappa(\rho_{\zeta},L,A)$, for any continuous $(L,A)$-quasi-geodesic $\gamma_n$ joining $o$ to $q_{n}$. That will show $q_{n}\rightarrow \zeta$.\par
 If $r\leq\kappa(\rho_{\zeta},L,A)$ then $ d(\alpha^{\zeta}(r),\gamma_n)\leq d(\alpha^{\zeta}(r),o)=r\leq \kappa(\rho_{\zeta},L,A)$ for
 all numbers $n$. Hence, in this case, $d(\gamma_n, N^{c}_{r}o\cap \alpha^{\zeta})\leq \kappa(\rho_{\zeta},L,A)$ for all numbers $n$.

Given $r\geq1$, only interesting pairs $(L,A)$ are those for which $r> \kappa(\rho_{\zeta},L,A)$. Definition \ref{Morse Constant} of $\kappa(\rho_{\zeta},L,A)$ gives $3L^{2},3A\leq \kappa(\rho_{\zeta},L,A)< r$ i.e. $L<\sqrt\frac{r}{3},A<\frac{r}{3}$. Set
 \begin{equation}
 R_{1}:=\underset{L<\sqrt\frac{r}{3},A<\frac{r}{3}}{\sup}\{2\kappa'(\rho_{\zeta},2L+1,A)+\lambda(\rho_{\zeta},L,A)\}< \infty,
 \end{equation}
 where $\lambda$ is the function as in Lemma \ref{CM 4.6}.\\
 Let $R=r+R_1$. From (i) we get a number $N=N(R)$ such that for all $n\geq N$ we have
\begin{equation}
  d(\gamma_n,N^{c}_{R}o\cap \alpha^{\zeta})\leq \kappa'(\rho_{\zeta},2L+1,A)
 \end{equation}
 for every continuous $(L,A)$-quasi-geodesic $\gamma_n$ joining $o$ to $q_{n}$. Note that $N$ is independent of $L,A$. There exist points $z_n\in\gamma_n$ and
 $x_n\in \alpha^{\zeta}([R,\infty))$ such that $d(x_n,z_n)\leq \kappa'(\rho_{\zeta},2L+1,A)$. Let $y_n$ be the last point of $\alpha^{\zeta}$ between $o$ and $\alpha^{\zeta}(x_n)$ such that $d(y_n,\gamma_n)=\kappa(\rho_{\zeta},L,A)$.  From  Lemma \ref{CM 4.6},
 $$d(o,y_n)\geq R-2\kappa'(\rho_{\zeta},2L+1,A)-\lambda(\rho_{\zeta},L,A)\geq R-R_1=r.$$ Thus,
 $d(\gamma_n,N^{c}_{r}o\cap \alpha^{\zeta})\leq \kappa(\rho_{\zeta},L,A)$ for all $n\geq N$.

 \end{proof}

 \section{Approximate Barycenters}
 Let $X$ be a proper geodesic metric space.  Assume
$|\partial_{c}^{\mathcal{F}\mathcal{Q}}X|\geq 3$.
Let $(a,b,c)$ be a distinct triple in $ X\cup \partial_{c}^{\mathcal{F}\mathcal{Q}}X$. By
joining the points $a,b,c$
with geodesics, we have a $\triangle{(a,b,c)}$ whose vertices are
$a,b,c$. The $\triangle{(a,b,c)}$ is called
ideal triangle if $a,b,c\in\partial_{c}^{\mathcal{F}\mathcal{Q}}X$. If the
sides of the triangle $\triangle{(a,b,c)}$
are contracting geodesics, by taking the maximum of contracting functions, we can assume that
the sides of $\triangle{(a,b,c)}$
are $\rho$-contracting for a single sub-linear function $\rho$.
\begin{definition}
(Approximate barycenter of triangles):
Given $\delta\geq0$.
An element $x$ in  $X$ is said to be a $ \delta$-$barycenter$ for
a triangle $\triangle{(a,b,c)}$ in  $X\cup \partial_{c}^{\mathcal{F}\mathcal{Q}}X$
if  the distance of $x$ from the  sides of $\triangle{(a,b,c)}$  is at most $\delta$.
\end{definition}
\begin{lemma}(Lemma 11, Lemma 12 of \cite{Mousley})
Let $\triangle{(a,b,c)}$ be a triangle in $X\cup \partial_{c}^{\mathcal{F}\mathcal{Q}}X$ such that its sides
are $\rho$-contracting geodesics. There exists $\delta=\delta(\rho)\geq 0$
such that the set of $\delta$-barycenters of $\triangle{(a,b,c)}$ is non-empty and its diameter is bounded above by
some constant depending only
  on $\rho$.
 \end{lemma}

A geodesic metric  space $X$ is said to be $\delta$-$hyperbolic$ if every geodesic triangle in $X$ has a $\delta$-$barycenter$. $X$ is said to be $hyperbolic$ if it is $\delta$-$hyperbolic$ for some $\delta\geq0$. Hyperbolicity of a geodesic metric space can be also defined in terms of uniform ‘thin’ triangles. We refer the reader to see  Proposition 1.17 of Chapter III.H in \cite{bridson} and   Section 6 in \cite{Bowditchnotes} for different equivalent notions of hyperbolicity.   A finitely generated group $G$ is said to be \textit{hyperbolic group}
   if its Cayley graph with respect to some finite generating set is a hyperbolic metric space.

   \section{Main result}
   Let $G$ be a finitely generated group $G$. Cashen-Mackay in \cite{cashen2017}
proved that the contracting boundary $\partial_{c}^{\mathcal{F}\mathcal{Q}}G$
   is a metrizable space.
   Let $H\leq G$ be a    finitely generated subgroup such that
$\Lambda(H)\subseteq\partial_{c}^{\mathcal{F}\mathcal{Q}}G$.
   Then the limit set $\Lambda(H)$ is also metrizable by equipping it with  subspace
topology inherited from $\partial_{c}^{\mathcal{F}\mathcal{Q}}G$.
   Let us assume that $\Lambda(H)$ contains three distinct points. Then the space
$\Theta_{3}(\Lambda(H))$ of distinct triples of $\Lambda(H)$
   is non-empty.

   \begin{theorem}\label{theorem}
   If $\Lambda(H)\subseteq \partial_{c}^{\mathcal{F}\mathcal{Q}}G$ is compact and
$\Lambda(H)$ contains at least three distinct points,
   then the action of the subgroup $H$ on $\Theta_{3}(\Lambda(H))$ is properly
discontinuous.
   \end{theorem}
   \begin{proof}
  Towards contradiction we assume that the action of the subgroup $H$ on $\Theta_{3}(\Lambda(H))$ is not properly
discontinuous. Then there exists a compact set $K\subseteq
\Theta_{3}(\Lambda(H))$  and a sequence $\{h_{n}\}$ of distinct elements of $H$ such that $h_{n}K\cap K\neq \emptyset$.
This implies that there exists sequence of points $\{(a_{n},b_{n},c_{n})\}$ and $\{(a'_{n},b'_{n},c'_{n})\}$  in $K$ such that
$h_{n}(a_{n},b_{n},c_{n})$$=$$(a'_{n},b'_{n},c'_{n})$. Since $\Theta_{3}(\Lambda(H))$ is metric space and $K\subseteq \Theta_{3}(\Lambda(H))$ is compact, by sequential compactness of $K$, $\{(a_{n},b_{n},c_{n})\}$ subsequentially converges to a point say $(a,b,c)$$\in K$. Also $\{(a'_{n},b'_{n},c'_{n})\}$
subsequentially converging to $(a',b',c')\in K$. So, after passing to a subsequence, we can assume $(a_{n},b_{n},c_{n})\rightarrow (a,b,c)$ and $(a'_{n},b'_{n},c'_{n})\rightarrow(a',b',c')$.\\
Let $\triangle{(a,b,c)}$, $\triangle{(a',b',c')}$ be ideal triangles corresponding to the points $(a,b,c)$ and $(a',b',c')$ respectively.
We can take the sides of $\triangle{(a,b,c)}$ and $\triangle{(a',b',c')}$  to be
uniformly contracting (take maximum of contracting functions of sides of the triangles).
So we have points say $B_{(a,b,c)}$
and $B_{(a',b',c')}$ in the Cayley graph of $G$ such that these points are $\delta$-approximate barycenters of $\triangle{(a,b,c)}$
and $\triangle{(a',b',c')}$ respectively for some $\delta\geq0$.
 The constant $\delta$ depends only on the contraction functions of sides of $\triangle{(a,b,c)}$ and $\triangle{(a',b',c')}$.
 Now consider geodesic triangles $\triangle_{n}$ and $\triangle'_{n}$
corresponding to points $(a_{n},b_{n},c_{n})$ and $(a'_{n},b'_{n},c'_{n})$
respectively.
 As $(a_{n},b_{n},c_{n})\rightarrow(a,b,c)$,  by Corollary \ref{MainLemma}, there
exists $M$ such that $B_{(a,b,c)}$ and $B'_{(a,b,c)}$ are $\delta+M$ barycenter for the triangles $\triangle_{n}$ and $\triangle'_{n}$ respectively. Let $\delta'=\delta+M$.\\
\textit{Claim:} The sequence $\{h_{n}(B_{(a,b,c)})\}$ lie in a bounded set.\\ \\
\textit{Proof of Claim:}\\
Let $x_{n},y_{n}$ and $z_{n}$ be respective points on the sides
$[a_{n},b_{n}]$,$[b_{n},c_{n}]$ and $[a_{n},c_{n}]$ of triangles $\triangle_{n}$ such that $d(B_{(a,b,c)},x_{n})\leq\delta'$,
$d(B_{(a,b,c)},y_{n})\leq\delta'$ and $d(B_{(a,b,c)},z_{n})\leq \delta'$. Suppose the sequence $\{h_{n}(B_{(a,b,c)})\}$ does not lie in
a bounded set. Then the sequence $\{h_{n}(x_{n})\}$ will also not lie in a bounded set.
  The point $h_{n}(x_{n})$ lies on a bi-infinite geodesic, say $\alpha_{n}$,
 joining points $a_{n}'$ and $b_{n}'$.  Also consider $p_n$ to be the nearest point
projection  of $o$ on $\alpha_{n}$.
 Take the path $\alpha_{n}'$ which is concatenation of any geodesic between $o$ to
$p_n$ and then the subsegment of $\alpha_{n}$ which contains $h_n(x_{n})$.
Each $\alpha_{n}'$ is a $(3,0)$ quasi-geodesic. By applying Lemma\ \ref{Convergence lemma}, passing to a subsequence if necessary, we get that
$h_n(x_{n})$ converges either to $a'$ or $b'$ in  $\partial_{c}^{\mathcal{F}\mathcal{Q}}G$.
Let us take $h_n(x_{n})\rightarrow$ $a'$. Similarly, after passing to a subsequence if necessary, we get that $h_n(y_{n})$
converges either to $b'$ or $c'$. But as $d(h_n(x_{n}),h_n(y_{n}))=d(x_n,y_n)\leq 2\delta'$  we must have either $a'= b'$ or $a'=c'$.
 If  $h_n(x_{n})\rightarrow b'$, after passing to a subsequence if necessary, we get that $h_n(z_{n})$
converges either to $a'$ or $c'$ and hence $b'=a'$ or $b'=c'$. This leads to a contradiction as we have assumed
$a'$, $b'$ and $c'$ to be distinct.
 Hence the claim.    \\

 Since the sequence $\{h_{n}\}$  was taken to be distinct, the above claim gives that
the sequence $\{h_{n}(B_{(a,b,c)})\}$  is bounded  and since the space (Cayley graph of
finitely generated group $G$)
 is proper this contradicts the fact that $H$ acts properly discontinuously on the
Cayley graph of $G$.\\
 Hence the subgroup $H$ acts properly discontinuously on $\Theta_{3}(\Lambda(H))$.
   \end{proof}

 Cashen-Mackay, in \cite{cashen2017}, proved that if a finitely generated group has compact contracting boundary then it is a hyperbolic group (See Theorem 10.1 of \cite{cashen2017}). In order to prove it,  Cashen-Mackay first proved that in a finitely generated group with compact contracting boundary, the geodesic rays are contracting and then they proved that the geodesic rays are uniformly contracting.
 The proof of geodesic rays uniformly contracting, given in \cite{cashen2017},  is hard and complicated. By assuming geodesic rays contracting in a finitely generated group with compact contracting boundary, we give a simpler way of proving the group to be hyperbolic.

\begin{theorem}\label{hyp theorem}
Let $G$ be a finitely generated group such that
$\partial_{c}^{\mathcal{F}\mathcal{Q}}G$ is compact. Then $G$ is a hyperbolic group.
\end{theorem}
\begin{proof}
Let $C_S(G)$ denote the Cayley graph of $G$ with respect to some finite generating set $S$. First note that the space
$ C_S(G)\cup \partial_{c}^{\mathcal{F}\mathcal{Q}}G$ is sequentially compact : let $\{x_n\}$ be a sequence of points in $G\cup C_S(G)$. If for infinitely many $n$'s,
$x_n\in \partial_{c}^{\mathcal{F}\mathcal{Q}}G$ then there exists a convergent subsequence of $\{x_n\}$, as $\partial_{c}^{\mathcal{F}\mathcal{Q}}G$ is compact metrizable space.
So, let us assume $\{x_n\}\in C_S(G)$ and $d(e,x_n)\to\infty$.  Let $\gamma_n$ be a geodesic in $C_S(G)$ joining $e$ to $x_n$.
Since $C_S(G)$ is proper, by Arzela-Ascoli theorem, $\{\gamma_n\}$ has a subsequence, $\{\gamma_{n_k}\}$, converging uniformly on compact sets to a geodesic ray, say $\gamma$, starting from $e$. The geodesic $\gamma$ is contracting (Theorem 10.1 of \cite{cashen2017}).
Let $\gamma$ represent the point $\zeta$ in the contracting boundary. Then $x_{n_k}\to \zeta$.

Suppose $G$ is not hyperbolic group then there exists
a sequence of unbounded positive numbers $(\delta_n)$ and a
sequence of triangles $\{\triangle(x_n,y_n,z_n)\}$ in $C_S(G)$ such that the triangle $\triangle(x_n,y_n,z_n)$ has a  $\delta_n$-barycenter in $G$ where  $\delta_n$ is minimal i.e. if
 $\delta'_n<\delta_n$ then the $\triangle(x_n,y_n,z_n)$ has no $\delta'_n$-barycenter. Let $q_n\in G$ be a $\delta_n$-barycenter
of the triangle $\triangle(x_n,y_n,z_n)$. After multiplication with $q_n^{-1}$ barycenters of $\triangle(x_n,y_n,z_n)$ will come at $e$, the identity of $G$. So, without loss of generality, we can assume  $e$ to be a $\delta_n$-barycenter of the triangle $\triangle(x_n,y_n,z_n)$. Let $\gamma^1_n$ be a geodesic joining $x_n$ and $y_n$ , $\gamma^2_n$ be a geodesic joining $y_n$ and $z_n$ and $\gamma^3_n$ be a geodesic joining $z_n$ and $x_n$.   Let $p^1_n$, $p^2_n$ and $p^3_n$ be a nearest point projection from $e$ onto sides $\gamma^1_n$, $\gamma^2_n$ and $\gamma^3_n$ respectively. Let $\alpha^{1}_{x_n}$$:=$$[e,p^1_n]*{\gamma^1_n}_{[p^1_n,x_n]}$, $\alpha^{1}_{y_n}$$:=$$[e,p^{1}_n]*{\gamma^{1}_n}_{[p^1_n,y_n]}$, similarly we define $\alpha^{2}_{y_n}$, $\alpha^{2}_{z_n}$, $\alpha^{3}_{z_n}$ and $\alpha^{3}_{x_n}$. Here “$*$” means concatenation of two paths. All these paths are continuous $(3,0)$ quasi-geodesic. Since $\delta_n\rightarrow \infty$, at least one of the sequence among $\{d(e,x_n)\}$, $\{d(e,y_n)\}$ and $\{d(e,z_n)\}$ is unbounded. If exactly one sequence is unbounded and rest are bounded then sequences $\{d(e,\gamma^1_n)\}$, $\{d(e,\gamma^2_n)\}$ and $\{d(e,\gamma^3_n)\}$ are bounded, contradiction to $\delta_n\rightarrow \infty$. \\
\textbf{Case 1:}  The sequences $\{d(e,y_n)\}$ and $\{d(e,z_n)\}$ are unbounded while $\{d(e,x_n)\}$ is bounded. The space $C_S(G)\cup\partial_c^{\mathcal{F}\mathcal{Q}}G$ is sequentially compact, hence the sequences $\{y_n\}$ and $\{z_n\}$ have subsequence converging to points in $\partial_{c}^{\mathcal{F}\mathcal{Q}}G$. Assume, without loss generality, that $y_n\rightarrow \zeta_1$ and $z_n\rightarrow \zeta_2$ in $\partial_c^{\mathcal{F}\mathcal{Q}}G$. If $\zeta_1\neq\zeta_2$ then by application of Theorem \ref{Mtheorem} gives that $\{d(e,\gamma^2_n)\}$ is bounded. Since $\{d(e,x_n)\}$ is bounded, we get that $\{d(e,\gamma^1_n)\}$, $\{d(e,\gamma^2_n)\}$, $\{d(e,\gamma^3_n)\}$ are bounded, which gives contradiction to $\delta_n\rightarrow \infty$.\par
Suppose $\zeta_1=\zeta_2=\zeta$. Given $r$, by definition of $x_n,y_n \rightarrow \zeta$, there exists $N=N(r)$ such that for $n\geq N$ \;
\begin{equation}
d\big(\alpha^{1}_{y_n},N^c_ro\cap \alpha^{\zeta}\big),d\big(\alpha^{2}_{y_n},N^c_ro\cap \alpha^{\zeta}\big),d\big(\alpha^{3}_{z_n},N^c_ro\cap \alpha^{\zeta}\big)\leq \kappa(\rho_{\zeta},3,0)
\end{equation}
There exists $r_i\geq r$ and $s_i>0$, where $i=1,2,3$, such that all the distances $d(\alpha^{\zeta}(r_1),\alpha^{1}_{y_n}(s_1)),
  d(\alpha^{\zeta}(r_2),\alpha^{2}_{y_n}(s_2)), d(\alpha^{\zeta}(r_3),\alpha^{3}_{z_n}(s_3))$ are at most $ \kappa(\rho_{\zeta},3,0)$.
  Let $r_m$$=$$\min\{r_1,r_2,r_3\}$. As $\alpha^{\zeta}$ is Morse, therefore there exists a positive number $\kappa_1=\kappa_1(\rho_{\zeta})$ and $t_i\leq s_i$
  such that the distances $d(\alpha^{\zeta}(r_m),\alpha^{1}_{y_n}(t_1)) $,  $d(\alpha^{\zeta}(r_m),\alpha^{2}_{y_n}(t_2))$, 
   $ d(\alpha^{\zeta}(r_m),\alpha^{3}_{z_n}(t_3))$ are at most $\kappa_1$.
  The lengths of $\alpha^{1}_{y_n}|_{[0,t_1]}$, $\alpha^{2}_{y_n}|_{[0,t_2]}$, $\alpha^{3}_{z_n}|_{[0,t_3]}$ tend to infinity
  as $r\to\infty$.
  By choosing $r$ large enough we have $d\big(\alpha^{\zeta}(r_m),\gamma^i_n\big)< \delta_n$ for all $n\geq N(r)$ and for all $i\in\{1,2,3\}$. This contradicts the minimality of $\delta_n$ for $n\geq N(r)$. Therefore $\zeta_1\neq\zeta_2$.\\

\textbf{Case 2:} Suppose all the three sequences $\{d(e,y_n)\}$, $\{d(e,z_n)\}$, $\{d(e,x_n)\}$ are unbounded. As in previous case, assume that $x_n\rightarrow \zeta_1$, $y_n\rightarrow \zeta_2$ and $z_n\rightarrow \zeta_3$ in $\partial_c^{\mathcal{F}\mathcal{Q}}G$. If $\zeta_1\neq \zeta_2\neq\zeta_3\neq \zeta_1$ then it will contradict, as in Case 1,  $\delta_n\rightarrow \infty$.\par
Suppose $\zeta_1=\zeta_2=\zeta$. Given $r\geq 1$, by definition of $x_n,y_n\rightarrow \zeta$, there exists $N=N(r)$ such that for $n\geq N$
\begin{equation}
d\big(\alpha^{1}_{x_n},N^c_ro\cap \alpha^{\zeta}\big),d\big(\alpha^{2}_{y_n},N^c_ro\cap \alpha^{\zeta}\big),d\big(\alpha^{3}_{y_n},N^c_ro\cap \alpha^{\zeta}\big)\leq \kappa(\rho_{\zeta},3,0)
\end{equation}
From the above equation, we get a large $r$ such that for all $n\geq N(r)$, the distance of $\alpha^{\zeta}(r)$ from all three sides of $\triangle_n$ is strictly less than $\delta_n$. This contradicts the minimality of $\delta_n$ for $n\geq N(r)$, hence $\zeta_1\neq\zeta_2$. Similarly one can show that $\zeta_2\neq\zeta_3$ and $\zeta_3\neq \zeta_1$.

\end{proof}

\end{document}